\documentclass[preprint,11pt]{elsarticle}
\usepackage{mathrsfs}
\usepackage{amsfonts}
\usepackage{amsmath}
\usepackage{stmaryrd}
\usepackage{amssymb}
\usepackage{amsthm}
\usepackage{mathrsfs}
\usepackage{amsfonts}
\usepackage{amscd}
\usepackage{indentfirst}
\usepackage{enumerate}
\usepackage{bookmark}
\allowdisplaybreaks      %多行公式自动分页
\def\p{\partial}
\def\R{\mathbb{R}}

\def\leq{\leqslant}

\newtheorem{theorem}{Theorem}[section]
\newtheorem{lemma}{Lemma}[section]

\usepackage{hyperref}
\begin{document}
\begin{frontmatter}

\title{A free boundary problem for a class of parabolic-elliptic type chemotaxis model
\tnoteref{sup}} \tnotetext[sup]{This work is supported by National Natural Science Foundation of China (Grant No. 11131005) and the Fundamental Research Funds for the Central Universities (Grant No. 2014201020202).}

%% use optional labels to link authors explicitly to addresses:
%% \author[label1,label2]{}
%% \address[label1]{}
%% \address[label2]{}

\author[whu]{Hua Chen}
\ead{chenhua@whu.edu.cn}

\author[sxu]{Wenbin Lv}
\ead{lvwenbin@whu.edu.cn}

\author[whu]{Shaohua Wu}
\ead{wush8@sina.com}

\address[whu]{School of Mathematics and Statistics, Wuhan University,
              Wuhan 430072, China}

\address[sxu]{School of Mathematical Sciences, Shanxi University, Taiyuan 030006, China}

\begin{abstract}
In this paper, we study a free boundary problem for a class of
parabolic-elliptic type chemotaxis model in high dimensional symmetry domain $\Omega$. By
using the contraction mapping principle and operator semigroup approach, we
establish the existence of the solution for such kind of chemotaxis system in
the domain $\Omega$ with free boundary condition. Besides, we get the explicit formula for the free boundary and show the chemotactic collapse for the solution of the system.
\end{abstract}

\end{frontmatter}

\section{Introduction}

Understanding of the partially oriented movement of cells in response to chemical signals, chemotaxis, is of great significance in various contexts. This importance partly stems from the fact that when cells combined with the ability to produce the respective signal substance themselves, chemotaxis mechanisms are among the most primitive forms of intercellular communication. Typical examples include
aggregation processes such as slime mold formation in Dictyostelium Discoideum discovered by K.B. Raper \cite{K.B.R.}. A lot of mathematicians have made efforts to develop various models and investigated them from a viewpoint of mathematics. For a broad overview over various types of chemotaxis processes, we refer the reader to the surveys and papers \cite{C-W20071,C-W2008,C-W2011,C-Z2004,C-Z2005,C-Z2006,H-P,D.H.,K-S,W-Y,W-C-L2008,Y-C-L-S} and the references therein.

As we all know, in a standard setting for many partial differential equations, we usually assume that the process being described occurs in a fixed domain of the space. But in the real world, the following phenomenon may happen. At the initial state, a kind of amoeba occupy some areas. When foods become rare, they begin to secrete chemical substances on their own. Since the biological time scale is much slower than the chemical one, the chemical substances are full filled with whole domain and create a chemical gradient attracting the cells. In turn, the areas of amoeba may change according to the chemical gradient from time to time. In other words, a part of whose boundary is unknown in advance and that portion of the boundary is called a free boundary. In addition to the standard boundary conditions that are needed in order to solve the PDEs, an additional condition must be imposed at the free boundary. One then seeks to determine both the free boundary and the solution of the differential equations. The theory of free boundaries has seen great progress in the last thirty years; for the state of the field we refer to \cite{A.F.}.

In this paper, we consider the following high dimensional free boundary problem
of a chemotaxis model. Such kind of models can be found in \cite{C-L-W2015,C-L-W2016,C-W2007,C-W2012,W2010}, and we can
give more explanations in the appendix below.
\begin{equation}\label{19}
\begin{cases}
u_{t}=\nabla(\nabla u-u\nabla v),&\text{in}\quad\Omega_{t}\times(0,T),\\
u=0,&\text{in}\quad\Omega\times(0,T)\setminus\Omega_{t}\times(0,T),\\
\displaystyle -\nabla u\cdot\frac{\nabla\Phi}{|\nabla\Phi|}=ku,&\text{on}\quad\partial\Omega_{t}\times(0,T),\\
\displaystyle u\frac{\partial\Phi}{\partial t}=\nabla u\cdot\nabla\Phi-u\nabla v\cdot\nabla\Phi,&\text{on}\quad\partial\Omega_{t}\times(0,T),\\
u(x,0)=u_{0}(x),&\text{in}\quad\Omega_{0},\\
0=\Delta v+u-c,&\text{in}\quad\Omega\times(0,T),\\
\displaystyle \frac{\partial v}{\partial \mathbf{n}}=0,&\text{on}\quad\partial\Omega\times(0,T),
\end{cases}
\end{equation}
where
\begin{itemize}
  \item $\Omega\subset\mathbb{R}^n$ is a bounded open set with smooth boundary $\partial\Omega$ and $\mathbf{n}$ is unit outer normal vector of $\partial\Omega$. Besides, $\Omega$ is assumed as a symmetry domain, i.e. if $x\in\Omega$ then $-x\in\Omega$ as well;
  \item $c$ is a constant which will be determined in the following section;
  \item $k(x,t)=k(|x|,t)$ is radial symmetric and satisfying the Lipschitz condition on $|x|$, namely there exists a constant $L>0$, such that
\begin{equation}\label{103}
\big|k(|x_1|,t)-k(|x_2|,t)\big|\leqslant L\big||x_1|-|x_2|\big|,~x_1,x_2\in\Omega_t.
\end{equation}
Also, $k(x,t)$ is bounded on $t\in[0,+\infty)$. In other words, there exists a constant $C>0$ which depends on $x$, such that
\begin{equation}\label{105}
|k(x,\,t)|\leqslant C,~t\in[0,+\infty);
\end{equation}
  \item $u=u(x,\,t)$ is an unknown function of $(x,\,t)\in\Omega_t\times(0,\,T)$ and it stands for the density of cellular slime molds. In other words, the density $u(x,\,t)$ occupies the domain $\Omega_t$, an open subset of $\Omega$, in time $t$ and $u(x,t)=0$ in the outside of $\Omega_t$;
  \item $v=v(x,\,t)$ is an unknown function of $(x,\,t)\in\Omega\times(0,\,T)$ and it stands for the concentration of chemical substances secreted by the slime molds;
  \item $\Gamma_t:~\Phi(x,t)=0$ is an unknown free boundary.
\end{itemize}
For general smooth domain $\Omega$, the system (\ref{19}) is based on the well-known chemotaxis model with fixed boundary
\begin{equation}\label{20}
\begin{cases}
u_t=\nabla(\nabla u-u\nabla v),&\text{in}\quad\Omega\times(0,\,T),\\
0=\Delta v+u-c,&\text{in}\quad \Omega\times(0,\,T),\\
\displaystyle \frac{\partial u}{\partial \mathbf{n}}=\frac{\partial v}{\partial \mathbf{n}}=0,&\text{on}\quad\partial\Omega\times(0,T),\\
u(x,0)=u_{0}(x),&\text{in}\quad\Omega.
\end{cases}
\end{equation}
which was introduced by W. J\"{a}ger and S. Luckhaus~\cite{J-L} who first studied the conjectrues of Childress and Percus in 1992. Later, the problem \eqref{20} is intensively studied by many authors (see for instance
\cite{C-M,M.H.,H-M-V1997,H-M-V1998,J-L}). Concerning the problem, classical results state that
\begin{itemize}
  \item if n=2
   \begin{itemize}
     \item there exists a critical number $C(\Omega)$ such that $c\leqslant C(\Omega)$ implies that there exists a unique, smooth positive solution to \eqref{20} globally in time \cite{J-L};
     \item $\Omega$ is a disk, there exists a positive number $C$ such that $c\geqslant C$ implies radially symmetric positive initial values can be constructed such that explosion of the solution happens in the center of the disc in finite time \cite{J-L};
     \item $\Omega=\R^2$, then no radial, self-similar solutions of \eqref{20} exist such that $\int_{|x|\leqslant r}u(T,s)ds<\infty$ as $r\rightarrow0$ \cite{M.H.,H-M-V1997};
   \end{itemize}
  \item if $n\geqslant 3$
   \begin{itemize}
     \item $\Omega=\R^3$ for any $T>0$ and any constant $C>0$, the radial solution $(u,v)$ of \eqref{20} that is smooth for all times $0<t<T$, blows up at $r=0$ and $t=T$, and is such that $\int_{|x|\leqslant r}u(T,s)ds\rightarrow C$ \cite{M.H.,H-M-V1997,H-M-V1998};
     \item  $\Omega$ is a ball and some additional conditions, there exists a unique maximal classical solution with finite maximal existence time $T$ \cite{C-M}.
   \end{itemize}
\end{itemize}

In one dimensional case, if $k$ is a positive constant H. Chen and S.H. Wu \cite{C-W2007,C-W2012,W2010} studied the similar free boundary value problem (\ref{19}) and established the existence and uniqueness of the solution for the system (\ref{19}). However, to the best of our knowledge, high dimensional case for the free boundary value problem (\ref{19}) will be more important. In view of the biological relevance of the particular case $n=3$, we find it worthwhile to clarify these questions. In the present paper, we consider the system (\ref{19}) in a high dimensional symmetry domain $\Omega$. In addition, the condition that $k$ is a positive constant in \cite{C-W2007,C-W2012,W2010} seems too strict. So it is also worthwhile to consider the system with non-constant coefficient $k$.

This paper is arranged as follows. In section \ref{sec2}, we rewrite the model and present the main result of the paper. In section \ref{sec3}, we use the operator semigroup approach to establish some estimates which are essential in the proof of the main result. In section \ref{sec4}, we shall give the proof of the main result. In addition, we study the property of the free boundary in section \ref{sec5}.

\section{Main result}\label{sec2}

\subsection{Rewrite the model}

In this subsection, we rewrite the model with the form of radial symmetry. We assume that the environment and solution are radially symmetric. Without loss of generality, we assume that $\Omega=B_1(0)$ which represents a unite ball centered in origin and that $u$ and $v$ are radially symmetric with respect to $x=0$. The free boundary can be written as $r=|x|=h(t)$ and $h(0)=b$. Hence, the density $u(x,0)$ occupies the domain $\Omega_0=B_b(0)$ at the initial time $t=0$.

Let
$(\widetilde{u},\widetilde{v})$ denote the corresponding radial solution in $B_1(0)\times(0,T)$. In order to avoid confusion, we may write
\begin{equation*}
\widetilde{u}(r,t)=u(x,t),\quad \widetilde{v}(r,t)=v(x,t),
\end{equation*}
 with $r=|x|=(x_1^2+\cdots+x_n^2)^{\frac{1}{2}}\in (0,1)$ and
$$\frac{\partial r}{\partial x_i}=\frac{\partial}{\partial x_i}\left((x_1^2+\cdots+x_n^2)^{\frac{1}{2}}\right)=\frac{x_i}{r},\quad i=1,2,\cdots,n.$$
A simple calculation shows that
\begin{equation}\label{21}
u_{x_i}(x,t)=\widetilde{u}_r(r,t)\frac{x_i}{r},\quad v_{x_i}=\widetilde{v}_r(r,t)\frac{x_i}{r},
\end{equation}
\begin{equation}\label{24}
u_{x_ix_i}(x,t)=\widetilde{u}_{rr}(r,t)\frac{x_i^2}{r^2}+\widetilde{u}_r(r,t)\left(\frac{1}{r}-\frac{x_i^2}{r^3}\right),
\end{equation}
and
\begin{equation}\label{25}
v_{x_ix_i}(x,t)=\widetilde{v}_{rr}(r,t)\frac{x_i^2}{r^2}+\widetilde{v}_r(r,t)\left(\frac{1}{r}-\frac{x_i^2}{r^3}\right).
\end{equation}

\noindent $\bullet$ Reformulation of the boundary.

If $\Gamma_t: \widetilde{\Phi}(r,t)=0\Leftrightarrow r-h(t)=0$, then the condition of the free boundary converts into
\begin{equation*}
\widetilde{k}\big(h(t),t\big)\widetilde{u}\big(h(t),t\big)+\widetilde{u}_r\big(h(t),t\big)=0,
\end{equation*}
and
\begin{equation*}
-\widetilde{u}\big(h(t),t\big)h_t(t)=\widetilde{u}_r\big(h(t),t\big)-\widetilde{u}\big(h(t),t\big)\widetilde{v}_r\big(h(t),t\big).
\end{equation*}

Actually, substituting (\ref{21}) and
$$\Phi_{x_i}(x,t)=\frac{\partial \widetilde{\Phi}}{\partial r}\cdot\frac{\partial r}{\partial x_i}=\frac{x_i}{r},\quad i=1,2,\cdots,n,$$
into the third and forth equations of (\ref{19}), we can easily get
$$\widetilde{k}(r,t)\widetilde{u}(r,t)=-\nabla u\cdot\frac{\nabla\Phi}{|\nabla\Phi|}=-\sum_{i=1}^n\widetilde{u}_r(r,t)\frac{x_i^2}{r^2}=-\widetilde{u}_r(r,t),$$
and
\begin{equation}\label{26}
\begin{aligned}
-\widetilde{u}(r,t)h_t(t)=&u\frac{\partial\Phi}{\partial t}=\nabla u\cdot\nabla\Phi-u\nabla v\cdot\nabla\Phi=\widetilde{u}_r(r,t)-\widetilde{u}(r,t)\widetilde{v}_r(r,t),
\end{aligned}
\end{equation}
on $\Gamma_t$.

If $u>0$ on $\Gamma_t$, then (\ref{26}) is equivalent to
\begin{equation*}
h_t(t)=\widetilde{k}\big(h(t),t\big)+\widetilde{v}_r\big(h(t),t\big).
\end{equation*}

\noindent $\bullet$ Reformulation of the equation.

Substituting (\ref{21}), (\ref{24}) and (\ref{25}) into the first and sixth equations of (\ref{19}), we can easily obtain
\begin{align*}
\widetilde{u}_{t}(r,t)&=u_{t}(x,t)=\Delta u-\nabla u\nabla v-u\Delta v\\
&=\widetilde{u}_{rr}(r,t)+\frac{n-1}{r}\widetilde{u}_r(r,t)-\widetilde{u}_r(r,t)\widetilde{v}_r(r,t)\\
&\quad\quad-\widetilde{u}(r,t)\left(\widetilde{v}_{rr}(r,t)+\frac{n-1}{r}\widetilde{v}_r(r,t)\right)\\
&=r^{1-n}\left(r^{n-1}\widetilde{u}_r(r,t)\right)_r-r^{1-n}\left(r^{n-1}\widetilde{u}(r,t)\widetilde{v}_r(r,t)\right)_r,
\end{align*}
\begin{equation}\label{1-1}
\text{or}\quad r^{n-1}\widetilde{u}_{t}(r,t)=\big(r^{n-1}\widetilde{u}_r(r,t)-r^{n-1}\widetilde{u}(r,t)\widetilde{v}_r(r,t)\big)_r,
\end{equation}
and
\begin{align*}
0&=\Delta v+u-c\\
&=\widetilde{v}_{rr}(r,t)+\frac{n-1}{r}\widetilde{v}_r(r,t)+\widetilde{u}(r,t)-c\\
&=r^{1-n}\left(r^{n-1}\widetilde{v}_r(r,t)\right)_r+\widetilde{u}(r,t)-c,
\end{align*}
\begin{equation}\label{1-2}
\text{or}\quad 0=\big(r^{n-1}\widetilde{v}_r(r,t)\big)_r+r^{n-1}\widetilde{u}(r,t)-r^{n-1}c.
\end{equation}
Therefore, the model we are concerned here becomes
\begin{equation}\label{1}
\begin{cases}
\widetilde{u}_{t}(r,t)=\widetilde{u}_{rr}(r,t)+\dfrac{n-1}{r}\widetilde{u}_r(r,t)\\
\quad\quad\quad\quad-\widetilde{u}_r(r,t)\widetilde{v}_r(r,t)-\widetilde{u}(r,t)\left(\widetilde{v}_{rr}(r,t)\right.\\
\quad\quad\quad\quad\left.+\dfrac{n-1}{r}\widetilde{v}_r(r,t)\right),&0<r<h(t),\quad 0<t<T,\\
\widetilde{u}(r,t)=0,&h(t)<r<1,\quad 0<t<T,\\
\widetilde{u}_r(0,t)=0,&0<t<T,\\
\widetilde{u}_r(h(t),t)+\widetilde{k}(h(t),t)\widetilde{u}(h(t),t)=0,&0<t<T,\\
h_t(t)=\widetilde{k}(h(t),t)+\widetilde{v}_r(h(t),t),&0<t<T,\\
\widetilde{u}(r,0)=\widetilde{u}_{0}(r),&0<r<b,\\
0=\widetilde{v}_{rr}(r,t)+\dfrac{n-1}{r}\widetilde{v}_r(r,t)+\widetilde{u}(r,t)-c,&0<r<1,\quad 0<t<T,\\
\widetilde{v}_r(0,t)=0,&0<t<T,\\
\widetilde{v}_r(1,t)=0,&0<t<T,
\end{cases}
\end{equation}
which corresponds to the equation with normal coordinate
\begin{equation}\label{101}
\begin{cases}
u_{t}=\nabla(\nabla u-u\nabla v),&\text{in}\quad B_{h(t)}(0)\times(0,T),\\
u=0,&\text{in}\quad\big(B_1(0)\setminus \overline{B_{h(t)}(0)}\big)\times(0,T),\\
\displaystyle \frac{\p u}{\p \mathbf{n}}(x,t)+k(x,t)u(x,t)=0,&\text{on}\quad\partial B_{h(t)}(0)\times(0,T),\\
\displaystyle h_t(t)=k(x,t)+\frac{\p v}{\p \mathbf{n}}(x,t),&\text{on}\quad\partial B_{h(t)}(0)\times(0,T),\\
u(x,0)=u_{0}(x),&\text{in}\quad B_b(0),\\
0=\Delta v+u-c,&\text{in}\quad B_1(0)\times(0,T),\\
\displaystyle \frac{\partial v}{\partial \mathbf{n}}(x,t)=0,&\text{on}\quad\partial B_1(0)\times(0,T).
\end{cases}
\end{equation}

\subsection{Main result}

Now we introduce the following space notations, which will be used in the main result. For $T>0$, we define
\begin{gather*}
X_{u}^{1}(T)=L^\infty\left((0,T],H^{1,p}\big(B_{h(t)}(0)\big)\cap\left\{\frac{\p u}{\p \mathbf{n}}(x,t)+k(x,t)u(x,t)=0,x\in\p B_{h(t)}(0)\right\}\right),\\
X_{v}^{2}(T)=L^\infty\left((0,T],H^{2,p}\big(B_1(0)\big)\cap\left\{\frac{\p v}{\p \mathbf{n}}(x,t)=0,x\in\p B_1(0)\right\}\right),\\
Y_v(T)=\left\{v~\bigg|~\frac{1}{|B_1(0)|}\int_{B_1(0)} vdx=0\right\}.
\end{gather*}

Our main results are:
\begin{theorem}\label{th1}
Assume $k(x,t)$ is radial symmetric on $x$ and satisfies the conditions (\ref{103}) and (\ref{105}) and
$$c=nM,\quad M=\frac{\Gamma(\frac{n}{2})}{2\pi^{\frac{n}{2}}}\int_{B_b(0)}u_0(x)dx,$$
where $n$ is the dimension of the space. If
$$u_{0}(x)\in H^{1,p}\big(B_b(0)\big),~u_{0}>0,$$
are radial symmetric on $x$, where $0<b<1$ and $b$ is a constant. Then there exist $T>0$ small enough, a radial symmetric pair
$$\big(u(x,t),v(x,t)\big)\in (X_u^{1}(T))\times (X_{v}^{2}(T)\cap Y_v(T)),$$
and a curve $\Gamma_t:|x|=h(t)$, which are the solutions of \eqref{101} for $n<p$.
\end{theorem}

Concerning the free boundary, we have the following properties:

\begin{theorem}\label{th2}
Assume $1<n<p$, $k(x,t)=a|x|$ and $a>0$ is a constant. Then we have $$h(t)=\left[\left(b^n-\frac{M}{a+M}\right)e^{n(a+M)t}+\frac{M}{a+M}\right]^{\frac{1}{n}}.$$
Furthermore, if $M<\frac{ab^n}{1-b^n}$, then $h(t)$ is increasing and
\[h(t)\rightarrow 1~\text{when}~t\rightarrow \frac{1}{n(a+M)}\ln\frac{a}{(a+M)b^n-M}.\]
If $M>\frac{ab^n}{1-b^n}$, then $h(t)$ is decreasing and
\[h(t)\rightarrow0~\text{as}~t\rightarrow \frac{1}{n(a+M)}\ln\frac{M}{(1-b^n)M-ab^n}.\]
If $M=\frac{ab^n}{1-b^n}$, then $h(t)\equiv b$.
\end{theorem}

Concerning the global existence and blow-up in finite time for the solution, we have the following result:

\begin{theorem}\label{th3}
Assume $1<n<p$, $k(x,t)=a|x|$ and $a>0$ is a constant. If $M\leqslant\frac{ab^n}{1-b^n}$, then the solution $(u,v)$ is a global solution. If $M>\frac{ab^n}{1-b^n}$, then the system \eqref{101} collapses in a finite time $T^*=\frac{1}{n(a+M)}\ln\frac{M}{(1-b^n)M-ab^n}$. More precisely
$$u(x,t)\rightarrow \frac{2\pi^{\frac{n}{2}}}{\Gamma(\frac{n}{2})}M\delta(x)~\text{as}~t\rightarrow T^*.$$
\end{theorem}

\section{Some crucial estimates}\label{sec3}

In this section, we establish some crucial estimates, which will play the key roles in proving the local existence of radial symmetric solution of the system \eqref{101} in high dimensional case.

\subsection{Some basic properties of the solution}

\begin{lemma}
If $u_0(x)>0$ and $(u,v)$ is the radial symmetrical solution of the system \eqref{101}, then we have $u>0$ and
$$\int_{B_{h(t)}(0)}u(x,t)dx=\int_{B_b(0)}u_0(x)dx.$$
\end{lemma}

\begin{proof}
Since $\widetilde{u}_{0}(r)>0$, by standard maximal principle of the parabolic equation, it follows that $\widetilde{u}>0$. Integrating the equation (\ref{1-1}) over $(0,h(t))$, we have
\begin{equation*}
\begin{split}
\int_0^{h(t)} r^{n-1}\widetilde{u}_{t}(r,t)dr=&\int_0^{h(t)}\left(r^{n-1}\widetilde{u}_r(r,t)-r^{n-1}\widetilde{u}(r,t)\widetilde{v}_r(r,t)\right)_rdr\\
=&h^{n-1}(t)\left[\widetilde{u}_r\big(h(t),t\big)-\widetilde{u}\big(h(t),t\big)\widetilde{v}_r\big(h(t),t\big)\right]\\
=&-h^{n-1}(t)\widetilde{u}\big(h(t),t\big)h_t(t),
\end{split}
\end{equation*}
where the fifth equation of (\ref{1}) is used. Thus one has
$$\frac{d}{dt}\int_0^{h(t)}r^{n-1}\widetilde{u}(r,t)dr=\int_0^{h(t)} r^{n-1}\widetilde{u}_{t}(r,t)dr+h_t(t)h^{n-1}(t)\widetilde{u}\big(h(t),t\big)=0,$$
which implies that
\begin{equation*}
\int_0^{h(t)}r^{n-1}\widetilde{u}(r,t)dr=\int_0^{h(0)}r^{n-1}\widetilde{u}(r,0)dr.
\end{equation*}
By virtue of the fact
$$\int_{B_b(0)}u_0(x)dx=\frac{2\pi^{\frac{n}{2}}}{\Gamma(\frac{n}{2})}\int_0^br^{n-1}\widetilde{u}(r,0)dr,$$
we have
$$\int_{B_{h(t)}(0)}u(x,t)dx=\int_{B_b(0)}u_0(x)dx$$
as required. Hence, the proof of the lemma is completed.
\end{proof}

\begin{lemma}\label{lm4}
We have
$$\widetilde{v}_r(h(t),t)=Mh(t)-\frac{M}{h^{n-1}(t)}.$$
\end{lemma}

\begin{proof}
Integrating the equation (\ref{1-2}) over $(0,h(t))$, we have
$$h^{n-1}(t)\widetilde{v}_r(h(t),t)=nM\int_0^{h(t)}r^{n-1}dr-\int_0^{h(t)}r^{n-1}\widetilde{u}(r,t)dr.$$
Simple calculation yields
$$\widetilde{v}_r(h(t),t)=Mh(t)-\frac{M}{h^{n-1}(t)}.$$
This completes the proof of Lemma \ref{lm4}.
\end{proof}

\subsection{Some basic properties of the system}

Firstly, we define
$$B=\left\{h\in C(0,T]~\bigg|~h(0)=b,\left|\frac{h(t_1)-h(t_2)}{t_1-t_2}\right|\leqslant M_0,t_1,t_2\in(0,T),t_1\ne t_2\right\},$$
where $M_0<\frac{1}{2}\frac{\min\{b,1-b\}}{T}$ is a constant.

In this section, we shall establish some estimates which are important in the proof of the main result. For any fixed $h(t)\in B$, we consider the following problems
\begin{equation}\label{5}
\begin{cases}
u_{t}=\nabla(\nabla u-u\nabla v),&\text{in}\quad B_{h(t)}(0)\times(0,T),\\
u=0,&\text{in}\quad\big(B_1(0)\setminus \overline{B_{h(t)}(0)}\big)\times(0,T),\\
\displaystyle \frac{\p u}{\p \mathbf{n}}(x,t)+k(x,t)u(x,t)=0,&\text{on}\quad\partial B_{h(t)}(0)\times(0,T),\\
u(x,0)=u_{0}(x),&\text{in}\quad B_b(0),
\end{cases}
\end{equation}
and
\begin{equation}\label{6}
\begin{cases}
0=\Delta v+u-c,&\text{in}\quad B_1(0)\times(0,T),\\
\displaystyle \frac{\partial v}{\partial \mathbf{n}}(x,t)=0,&\text{on}\quad\partial B_1(0)\times(0,T).
\end{cases}
\end{equation}

Concerning the system \eqref{6}, we have the following result.

\begin{lemma}\label{lm2}
If $h(t)\in B$, then for $T>0$ small enough and $u\in X_u^1(T)$, the system \eqref{6} admits a unique solution $v\in X_v^2(T)\cap Y_v(T)$, and we have
\begin{equation}\label{9}
\sup_{0\leqslant t\leqslant T}\left\|v\right\|_{H^{2,p}(B_1(0))}\leqslant C\sup_{0\leqslant t\leqslant T}\|u\|_{L^p(B_{h(t)}(0))}+C.
\end{equation}
where $C$ is independent of $T$ and $h(t)\in B$.
\end{lemma}

\begin{proof}
It is obvious that the problem (\ref{6}) has a solution $v\in H^{1,p}(B_1(0))$ if and only if $\int_{B_1(0)}(u-c)dx=0$ i.e.
\begin{align*}
\int_{B_1(0)}(u-c)dx=&\frac{2\pi^{\frac{n}{2}}}{\Gamma(\frac{n}{2})}\int_0^{h(t)}r^{n-1}\widetilde{u}(r,t)dr-c|B_1(0)|\\
=&\frac{2\pi^{\frac{n}{2}}}{\Gamma(\frac{n}{2})}M-c\frac{2\pi^{\frac{n}{2}}}{n\Gamma(\frac{n}{2})}\\
=&0.
\end{align*}
Moreover, we have
$$\sup_{0\leqslant t\leqslant T}\left\|v\right\|_{H^{2,p}(B_1(0))}\leqslant C\sup_{0\leqslant t\leqslant T}\|u\|_{L^p(B_{h(t)}(0))}+C.$$

From the above estimate, we can easily draw the conclusion.
\end{proof}

Now, we consider the system \eqref{5}. Let $y_i=\frac{x_i}{h(t)},~\mathfrak{u}(y,t)=u(x,t),~\mathfrak{v}(y,t)=v(x,t),~\mathfrak{k}(y,t)=k(x,t)$. Then the system \eqref{5} is equivalent to the following system
\begin{equation*}
\begin{cases}
\displaystyle\mathfrak{u}_t=\frac{1}{h^2(t)}\Delta\mathfrak{u}-\frac{1}{h^2(t)}\nabla\cdot(\mathfrak{u}\nabla\mathfrak{v})+\frac{ h'(t)}{h(t)}y\nabla\mathfrak{u},&\text{in}\quad y\in B_1(0),\quad0<t<T,\\
\displaystyle\frac{\partial\mathfrak{u}}{\partial \mathbf{n}}(y,t)+\mathfrak{k}\big(y,t\big)h(t)\mathfrak{u}(y,t)=0,&\text{on}\quad y\in\p B_1(0),\quad0<t<T,\\
\displaystyle\mathfrak{u}(y,0)=\mathfrak{u}_0(y),&\text{in}\quad y\in B_1(0).
\end{cases}
\end{equation*}
Hence, it can be reduced to the system of integral equation
\begin{equation*}
\mathfrak{u}(y,t)=e^{\alpha(t,0)\Delta}\mathfrak{u}(y,0)+\int_0^t e^{\alpha(t,s)\Delta}A_{\mathfrak{u}\mathfrak{v}}(s)ds,
\end{equation*}
where
$$\alpha(t,s)=\int_s^t \frac{1}{h^2(\tau)}d\tau,$$
\begin{equation*}
A_{\mathfrak{u}\mathfrak{v}}(s)=-\frac{1}{h^2(s)}\nabla\cdot(\mathfrak{u}\nabla\mathfrak{v})+\frac{ h'(s)}{h(s)}y\nabla\mathfrak{u}.
\end{equation*}
Let~$t\geqslant s>0$,~$e^{\alpha(t,s)\Delta}$~represents the operator semigroup on $L^{p}\big(B_1(0)\big)$ which is generated by
\begin{gather*}
L(\tau)=h^{-2}(\tau)\Delta,\\ D\big(L(\tau)\big)=H^{1,p}\big(B_1(0)\big)\cap\left\{\frac{\partial\mathfrak{u}}{\partial \mathbf{n}}(y,t)+\mathfrak{k}\big(y,t\big)h(t)\mathfrak{u}(y,t)=0,y\in\p B_1(0)\right\}.
\end{gather*}
and $e^{\alpha(t,s)\Delta}$ is an holomorphic semigroup on $L^{p}\big(B_1(0)\big)$. If~$h(t)\in B$, then
\begin{equation}\label{8}
0<b-M_0T\leqslant h(t)\leqslant b+M_0T<1,\quad 0\leqslant t\leqslant T.
\end{equation}
The operator-theoretic feature of $e^{\alpha(t,s)\Delta}$ necessary for the proof of this lemma is well known \cite{M.T.2013}. There exists a constant $C>0$ which depends on $M_0$ but independent of $t$ such that
\begin{equation}\label{12}
\|e^{\alpha(t,s)\Delta}f\|_{H^{\lambda+1,p}}\leqslant C(t-s)^{-\frac{1}{2}-\frac{n}{2}(\frac{1}{q}-\frac{1}{p})}\|f\|_{H^{\lambda,q}},\quad f\in H^{\lambda,q},
\end{equation}
where $1\leqslant q\leqslant p\leqslant+\infty$.

Concerning the system \eqref{5} and \eqref{6}, we have the following result.

\begin{lemma}\label{lm3}
If $h(t)\in B,~u_0\in H^{1,p}\big(B_b(0)\big)\cap\{\frac{\p u}{\p \bf{n}}(x,0)+k(x,0)u(x,0)=0,x\in \p B_b(0)\}$ and $v_0\in H^{2,p}\big(B_1(0)\big)\cap\{\frac{\p v}{\p \bf{n}}(x,0)=0,x\in\p B_1(0)\}$. Then for $T$ small enough, the system (\ref{5}) and (\ref{6}) admit a unique solution
$$u\in X_u^1(T),\quad v\in X_v^{2}(T)\cap Y_v(T).$$
Moreover, we have
$$\sup_{0\leqslant t\leqslant t_0}\|u\|_{H^{1,p}(B_{h(t)}(0))}\leqslant 2C\|u(\cdot,0)\|_{H^{1,p}(B_b(0))},$$
where $n<p$.
\end{lemma}

\begin{proof}
Firstly, we consider $w\in X_u^1(T)$ and $w(x,0)=u_0(x)$. Let $v=v(w)$ denote the corresponding solution of the equation
\begin{equation}\label{14}
\begin{cases}
0=\Delta v+w-c,&\text{in}\quad B_1(0)\times(0,T),\\
\displaystyle \frac{\partial v}{\partial \mathbf{n}}(x,t)=0,&\text{on}\quad\partial B_1(0)\times(0,T).
\end{cases}
\end{equation}
By Lemma \ref{lm2}, we have
$$\sup_{0\leqslant t\leqslant T}\left\|v\right\|_{H^{2,p}(B_1(0))}\leqslant C\sup_{0\leqslant t\leqslant T}\|w\|_{L^p(B_{h(t)}(0))}+C.$$

Secondly, for the solution $v$ of equation (\ref{14}), we define $u=u\big(v(w)\big)$ to be the corresponding solution of
\begin{equation*}
\begin{cases}
\displaystyle\mathfrak{u}_t=\frac{1}{h^2(t)}\Delta\mathfrak{u}-\frac{1}{h^2(t)}\nabla\cdot(\mathfrak{w}\nabla\mathfrak{v})+\frac{ h'(t)}{h(t)}y\nabla\mathfrak{w},&\text{in}\quad y\in B_1(0),\quad0<t<T,\\
\displaystyle\frac{\partial\mathfrak{u}}{\partial \mathbf{n}}(y,t)+\mathfrak{k}\big(y,t\big)h(t)\mathfrak{u}(y,t)=0,&\text{on}\quad y\in\p B_1(0),\quad0<t<T,\\
\displaystyle\mathfrak{u}(y,0)=\mathfrak{u}_0(y),&\text{in}\quad y\in B_1(0).
\end{cases}
\end{equation*}

Let
$$B(\mathcal{M},T)=\{w\in X_u^1\mid w(x,0)=u_0(x),~\sup_{0\leqslant t\leqslant T}\|\mathfrak{w}(\cdot,t)\|_{H^{1,p}(B_1(0))}\leqslant \mathcal{M}\},$$
and we define a mapping $Fw=u\big(v(w)\big)$ as follows
\begin{equation}\label{11}
\mathfrak{u}(y,t)=e^{\alpha(t,0)\Delta}\mathfrak{u}(y,0)+\int_0^t e^{\alpha(t,s)\Delta}A_{\mathfrak{w},\mathfrak{v}}(s)ds.
\end{equation}
Then local existence will be established via contraction mapping principle.

{\bf Step 1:} $F$ maps $B(\mathcal{M},T)$ into itself.

Noticing \eqref{11}, we have
\begin{equation}\label{13}
\begin{aligned}
\left\|\mathfrak{u}(\cdot,t)\right\|_{H^{1,p}}\leqslant&\left\|e^{\alpha(t,0)\Delta}\mathfrak{u}(\cdot,0)\right\|_{H^{1,p}}+\int_0^{t} \left\|e^{\alpha(t,s)\Delta}\frac{1}{h^2(s)}\nabla\cdot\big(\mathfrak{w}\nabla\mathfrak{v}\big)\right\|_{H^{1,p}} ds\\
&\quad\quad+\int_0^{t} \left\|e^{\alpha(t,s)\Delta}\frac{ h'(s)}{h(s)}y\nabla\mathfrak{w}\right\|_{H^{1,p}}ds\\
=&(S_1)+(S_2)+(S_3).
\end{aligned}
\end{equation}
Using the facts \eqref{8} and \eqref{12}, the terms of the right-hand side of \eqref{13} are estimated from above by
$$(S_1)\leqslant C\left\|\mathfrak{u}(\cdot,0)\right\|_{H^{1,p}}\leqslant C\left\|\mathfrak{u}_0\right\|_{H^{1,p}};$$
\begin{equation*}
\begin{aligned}
(S_2)&\leqslant C\int_0^{t}(t-s)^{-\frac{1}{2}}\left\|\mathfrak{w}\nabla\mathfrak{v}\right\|_{H^{1,p}}ds\\
&\leqslant Ct^{\frac{1}{2}}\sup_{0\leqslant s\leqslant T}\left\|\mathfrak{w}\right\|_{H^{1,p}}\cdot\sup_{0\leqslant s\leqslant T}\left\|\mathfrak{v}\right\|_{H^{2,p}};
\end{aligned}
\end{equation*}
and
$$(S_3)\leq C\int_0^{t}\left[1+(t-s)^{-\frac{1}{2}}\right]\left\|y\nabla\mathfrak{w}\right\|_{L^{p}}ds\leqslant C\left(t+t^{\frac{1}{2}}\right)\sup_{0\leqslant s\leqslant T}\left\|\mathfrak{w}\right\|_{H^{1,p}}.$$
For $w\in B(\mathcal{M},T)$ this implies
\begin{equation*}
\left\|\mathfrak{u}(\cdot,t)\right\|_{H^{1,p}}\leqslant C\left\|\mathfrak{u}_0\right\|_{H^{1,p}}+CT^{\frac{1}{2}}\mathcal{M}(\mathcal{M}+1)+C\left(T+T^{\frac{1}{2}}\right)\mathcal{M}.
\end{equation*}
If we take $\mathcal{M}>0$ as large as
$$C\left\|\mathfrak{u}_0\right\|_{H^{1,p}}\leqslant\frac{\mathcal{M}}{2},$$
and then take $T>0$ as small as
$$CT^{\frac{1}{2}}\mathcal{M}(\mathcal{M}+1)+C\left(T+T^{\frac{1}{2}}\right)\mathcal{M}\leqslant\frac{\mathcal{M}}{2},$$
it holds that
$$\sup_{0\leqslant s\leqslant T}\|\mathfrak{u}\|_{H^{1,p}}\leqslant \mathcal{M},$$
i.e. $F$ maps $B(\mathcal{M},T)$ into itself.

{\bf Step 2:} $F$ is a contract mapping.

For $w_1,w_2\in B(\mathcal{M},T)$, we can construct $v_1,v_2$ as above and it holds that
\begin{equation*}
\begin{cases}
0=\Delta (v_1-v_2)+(w_1-w_2),&\text{in}\quad B_1(0)\times(0,T),\\
\displaystyle \frac{\partial (v_1-v_2)}{\partial \mathbf{n}}(x,t)=0,&\text{on}\quad\partial B_1(0)\times(0,T).
\end{cases}
\end{equation*}
Similar to the proof of Lemma \ref{lm2}, we can obtain that for $T>0$ small enough
\begin{equation}\label{17}
\sup_{0\leqslant t\leqslant T}\left\|v_1-v_2\right\|_{H^{2,p}(B_1(0))}\leqslant C\sup_{0\leqslant t\leqslant T}\|w_1-w_2\|_{L^p(B_{h(t)}(0))}.
\end{equation}
Let $u_1$, $u_2$ denote the corresponding solution of \eqref{11} respectively. Then the difference $u_1-u_2$ satisfies
\begin{equation*}
\begin{cases}
\displaystyle(\mathfrak{u}_1-\mathfrak{u}_2)_t=\dfrac{1}{h^2(t)}\Delta(\mathfrak{u}_1-\mathfrak{u}_2)\\
\quad-\dfrac{1}{h^2(t)}\nabla\cdot[(\mathfrak{w}_1-\mathfrak{w}_2)\nabla\mathfrak{v}_1]\\
\displaystyle\quad-\frac{1}{h^2(t)}\nabla\cdot[\mathfrak{w}_2\nabla(\mathfrak{v}_1-\mathfrak{v}_2)]+\frac{ h'(t)}{h(t)}y\nabla(\mathfrak{w}_1-\mathfrak{w}_2),&\quad y\in B_1(0),\quad0<t<t_0,\\
\displaystyle\frac{\partial(\mathfrak{u}_1-\mathfrak{u}_2)}{\partial \mathbf{n}}(y,t)+\mathfrak{k}\big(y,t\big)h(t)(\mathfrak{u}_1-\mathfrak{u}_2)(y,t)=0,&\quad y\in\p B_1(0),\quad0<t<t_0,\\
\displaystyle(\mathfrak{u}_1-\mathfrak{u}_2)(y,0)=0,&\quad y\in B_1(0).
\end{cases}
\end{equation*}
We have
\begin{align*}
\|\mathfrak{u}_1-\mathfrak{u}_2\|_{H^{1,p}}\leqslant&\int_0^{t} \left\|e^{\alpha(t,s)\Delta}\frac{1}{h^2(s)}\nabla\cdot\left[(\mathfrak{w}_1-\mathfrak{w}_2)\nabla\mathfrak{v}_1\right]\right\|_{H^{1,p}} ds\\
&\quad+\int_0^{t} \left\|e^{\alpha(t,s)\Delta}\frac{1}{h^2(s)}\nabla\cdot\left[\mathfrak{w}_2\nabla(\mathfrak{v}_1-\mathfrak{v}_2)\right]\right\|_{H^{1,p}}\\
&\quad\quad+\int_0^{t} \left\|e^{\alpha(t,s)\Delta}\frac{ h'(s)}{h(s)}y\nabla(\mathfrak{w}_1-\mathfrak{w}_2)\right\|_{H^{1,p}}ds\\
\leqslant& C\int_0^{t}(t-s)^{-\frac{1}{2}}\left\|(\mathfrak{w}_1-\mathfrak{w}_2)\nabla\mathfrak{v}_1\right\|_{H^{1,p}}ds\\
&\quad+C\int_0^{t}(t-s)^{-\frac{1}{2}}\left\|\mathfrak{w}_2\nabla(\mathfrak{v}_1-\mathfrak{v}_2)\right\|_{H^{1,p}}ds\\
&\quad\quad+C\int_0^{t}\left[1+(t-s)^{-\frac{1}{2}}\right]\|y\nabla(\mathfrak{w}_1-\mathfrak{w}_2)\|_{L^{p}}ds.
\end{align*}
For $w_1,w_2\in B(\mathcal{M},T)$ this implies
\begin{align*}
\|\mathfrak{u}_1-\mathfrak{u}_2\|_{H^{1,p}}\leqslant& CT^{\frac{1}{2}}\sup_{0\leqslant s\leqslant T}\|\mathfrak{v}_1\|_{H^{2,p}}\sup_{0\leqslant s\leqslant T}\|\mathfrak{w}_1-\mathfrak{w}_2\|_{H^{1,p}}\\
&\quad+CT^{\frac{1}{2}}\sup_{0\leqslant s\leqslant T}\|\mathfrak{w}_2\|_{H^{1,p}}\sup_{0\leqslant s\leqslant T}\|\mathfrak{v}_1-\mathfrak{v}_2\|_{H^{2,p}}\\
&\quad\quad+C\left(T+T^{\frac{1}{2}}\right)\sup_{0\leqslant s\leqslant T}\|\mathfrak{w}_1-\mathfrak{w}_2\|_{H^{1,p}}\\
\leqslant& CT^{\frac{1}{2}}(\mathcal{M}+1)\sup_{0\leqslant s\leqslant T}\|\mathfrak{w}_1-\mathfrak{w}_2\|_{H^{1,p}}\\
&\quad+CT^{\frac{1}{2}}\mathcal{M}\sup_{0\leqslant s\leqslant T}\|\mathfrak{w}_1-\mathfrak{w}_2\|_{H^{1,p}}\\
&\quad\quad+C\left(T+T^{\frac{1}{2}}\right)\sup_{0\leqslant s\leqslant T}\|\mathfrak{w}_1-\mathfrak{w}_2\|_{H^{1,p}}.
\end{align*}
Taking $T>0$ as small as
$$CT^{\frac{1}{2}}(\mathcal{M}+1)+CT^{\frac{1}{2}}\mathcal{M}+C\left(T+T^{\frac{1}{2}}\right)\leqslant\frac{1}{2},$$
we obtain $F$ is a contract mapping. This completes the proof of the lemma.
\end{proof}

\section{The existence of the solution locally in time}\label{sec4}

Having at hand the preliminary material collected above, we are now prepared to prove our main
result on local in time existence of solutions. We prove Theorem \ref{th1} by using Lemma \ref{lm2} and Lemma \ref{lm3}.

For each $h(t)\in B$, by Lemma \ref{lm3} we know that there exists a pair
$$u\in X_u^1(T),\quad v\in X_v^{2}(T)\cap Y_v(T),$$
which is the solution of the system \eqref{5} and \eqref{6}.

The following, we will use the contraction mapping principle to show the existence of the free boundary $r=h(t)$.

Set
\begin{equation*}
g(t)=b+\int_0^tk(x,s)ds+\int_0^t\frac{\p v}{\p \mathbf{n}}(x,s)ds,\quad x\in \p B_{h(t)}(0),
\end{equation*}
then we have
\begin{equation}\label{18}
\begin{aligned}
\left|\frac{g(t_1)-g(t_2)}{t_1-t_2}\right|\leq&\left|\frac{1}{t_1-t_2}\right|\int_{t_1}^{t_2}\left|k\big(x,s\big)\right|ds+\left|\frac{1}{t_1-t_2}\right|\int_{t_1}^{t_2}\left|\frac{\p v}{\p \mathbf{n}}(x,s)\right|ds\\
\leqslant&\sup_{0\leqslant t\leqslant T}\|k(\cdot,t)\|_{L^\infty}+C\sup_{0\leqslant t\leqslant T}\|\nabla v(\cdot,t)\|_{L^\infty}\\
\leqslant&C+C\|u(\cdot,0)\|_{H^{2,p}(B_b(0))}.
\end{aligned}
\end{equation}
Let $M_1$ denote the constant at the right hand of (\ref{18}). If $T$ is small enough, then
$$\frac{1}{2}\frac{\min\{b,1-b\}}{T}>M_1.$$
Take $M_0=M_1$ in $B$, it is clear that $B\subset C[0,T]$.

Define $G:~h(t)\rightarrow g(t)$, therefore $G$ maps $B$ into itself. Next we will demonstrate
that $G$ is contractive. Then the fixed point theorem yields that there exist a pair $(u,v)$ and
a curve $\Gamma:~r=h(t)$ which are the solution of (\ref{1}).

For $h_1(t),~h_2(t)\in B$, let $(u_1,v_1),~(u_2,v_2)$ represent the corresponding solutions of \eqref{101} respectively, then
\begin{align*}
&G(h_1)-G(h_2)=g_1(t)-g_2(t)\\
=&\int_0^t\left(k(x,s)-k(y,s)\right)ds+\int_0^t\Big(\frac{\p v_1}{\p \mathbf{n}}(x,s)-\frac{\p v_2}{\p \mathbf{n}}(y,s)\Big)ds\\
=&(I_1)+(I_2),
\end{align*}
where $x\in\p B_{h_1(t)}(0),~y\in\p B_{h_2(t)}(0)$.

The first term $(I_1)$ on the right-hand side are estimated from above by
$$(I_1)\leqslant L\int_0^t\big|h_{1}(s)-h_{2}(s)\big|ds\leqslant Lt\sup_{0\leqslant s\leqslant T}|h_1-h_2|.$$
By Lemma \ref{lm4}, the second term $(I_2)$ on the right-hand side are estimated from above by
\begin{align*}
(I_2)=&\int_0^t\left|\frac{\p \widetilde{v_1}}{\p r}(h_1(s),s)-\frac{\p \widetilde{v_2}}{\p r}(h_2(s),s)\right|ds\\
\leqslant&\frac{c}{n}\int_0^t|h_1(s)-h_2(s)|ds+M\int_0^t\left|\frac{1}{h_1^{n-1}(s)}-\frac{1}{h_2^{n-1}(s)}\right|ds\\
\leqslant& Ct\sup_{0\leqslant s\leqslant T}|h_1-h_2|.
\end{align*}

Hence the fixed point theorem yields that there exist a pair $(u,v)$ and a curve $\Gamma:~r=h(t)\in B$ which are the solution of (\ref{1}).

\section{Proofs of Theorem \ref{th2} and Theorem \ref{th3}}\label{sec5}

In the section, we study the property of the free boundary and give the proofs for Theorem \ref{th2} and Theorem \ref{th3}.

\begin{proof}[The proof of Theorem \ref{th2}]
Let $(u,v)\in(X_u^{1}(T))\times (X_{v}^{2}(T)\cap Y_v(T))$ and a curve $\Gamma:~r=h(t)\in B$ be the solution of (\ref{1}). We have
\begin{equation*}
\begin{cases}
h_t(t)=\widetilde{k}(h(t),t)+\widetilde{v}_r(h(t),t),\\
h(0)=b.
\end{cases}
\end{equation*}
According to the assumption and Lemma \ref{lm4}, we get
\begin{equation*}
\begin{cases}
h_t(t)=(a+M)h(t)-\frac{M}{h^{n-1}(t)},\\
h(0)=b.
\end{cases}
\end{equation*}
Let $s(t)=h^{n}(t)$, then it holds that
\begin{equation*}
\begin{cases}
s_t(t)=n(a+M)s(t)-nM,\\
s(0)=b^n.
\end{cases}
\end{equation*}
A direct calculation gives the following unique solution:
$$s(t)=\left(b^n-\frac{M}{a+M}\right)e^{n(a+M)t}+\frac{M}{a+M}.$$
Hence, we obtain
$$h(t)=\left[\left(b^n-\frac{M}{a+M}\right)e^{n(a+M)t}+\frac{M}{a+M}\right]^{\frac{1}{n}}.$$
Thus, we can deduce all the results in Theorem \ref{th2} easily.
\end{proof}

\begin{proof}[The proof of Theorem \ref{th3}]
If $M=\frac{ab^n}{1-b^n}$, then $h(t)\equiv b$. Thus by the extensive method, we can get the global existence for the solution of the free boundary problem \eqref{101}.

If $M<\frac{ab^n}{1-b^n}$, then the result of Theorem \ref{th2} tells us that $h(t)$ is increasing and $h(t)\to 1$ when $t\to \frac{1}{n(a+M)}\ln\frac{a}{(a+M)b^n-M}$. Since we know that $0\leqslant h(t)\leqslant 1$, which means $h(t)\equiv 1$ when $t\geqslant\frac{1}{n(a+M)}\ln\frac{a}{(a+M)b^n-M}$. Thus, it is similar that
 we can get the global existence for the solution of the free boundary problem \eqref{101}.

Also from Theorem \ref{th2}, we know that if $M>\frac{ab^n}{1-b^n}$, then $0<h(t)<1$ for each $0<t<T^*=\frac{1}{n(a+M)}\ln\frac{M}{(1-b^n)M-ab^n}$. Then, applying the extensive method, we can assert that the problem \eqref{101} has a solution for each $t\in(0,T^*)$.

Notice that $h(t)\rightarrow 0$ as $t\rightarrow T^*$ and in the outside of $B_{h(t)}(0)$, the solution $u\equiv0$. Thus we have
$$u(x,t)\rightarrow \frac{2\pi^{\frac{n}{2}}}{\Gamma(\frac{n}{2})}M\delta(x)~\text{as}~t\rightarrow T^*,$$
which implies that the chemotaxis model \eqref{101} collapses in a finite time $t\to T^*=\frac{1}{n(a+M)}\ln\frac{M}{(1-b^n)M-ab^n}$. Theorem \ref{th3} is proved.
\end{proof}

\section{Appendix}

In this section, let us recall the construction of the system. All of the material here can be
found in \cite{C-W2012}.

Let $\Omega\subset \mathbb{R}^{n}$ be a bounded open domain and $\Omega_{0}\subset\Omega$ be an open sub-domain. Assume a population density $u(x,0)$ occupying the domain $\Omega_{0}$, and in the outside of $\Omega_{0}$ the population density $u(x,0)\equiv0$ and the external signal $v$ occupying $\Omega$. For $t>0$, $u(x,t)$ spreads to domain $\Omega_{t}\subset\Omega$. Let $\partial\Omega_{t}$ denote the boundary of $\Omega_{t}$ and $n_{t}$ denote the outer normal vector of $\partial\Omega_{t}$, then $\Gamma_{t}=\partial\Omega_{t}\times(0,T)$ is the free boundary.

The spatial diffusion of species is related to the free boundary of $\Omega_{t}$ at the time $t\geqslant 0$. Observe the flux is increasing with respect to the density of the species, so it would be reasonable to suppose that flux is proportional to the density. Thus we have following flux condition on $\partial\Omega_{t}$,
\begin{equation*}
\begin{split}
-\nabla u\cdot n_{t}=k(x,t)u&\qquad\text{on}\quad\partial\Omega_{t}
\end{split}
\end{equation*}
where $k(x,t)$ is positive function, and $\frac{1}{k(x,t)}>0$ is mass flow ratio.

On the other hand, noticing that the total flux on $\partial\Omega_{t}$ is
\begin{equation*}
\begin{split}
j=-\nabla u\cdot n_{t}+\chi u\nabla v\cdot n_{t},
\end{split}
\end{equation*}
By conservation of population, one has
\begin{equation}\label{1-6}
uv_{n_{t}}=-\nabla u\cdot n_{t}+\chi u\nabla v\cdot n_{t}\quad\text{on}\quad\partial\Omega_{t},
\end{equation}
where $v_{n_{t}}$ is the normal diffusion velocity of $\partial\Omega_{t}$. Assume $\Gamma_{t}:\Phi(x,t)=0$, then
\begin{equation}\label{1-3}
v_{n_{t}}=\left(\frac{dx_{1}}{dt},\frac{dx_{2}}{dt},\cdot\cdot\cdot,\frac{dx_{n}}{dt}\right)\cdot n_{t}=
\left(\frac{dx_{1}}{dt},\frac{dx_{2}}{dt},\cdot\cdot\cdot,\frac{dx_{n}}{dt}\right)\cdot\frac{\nabla\Phi}{|\nabla\Phi|},
\end{equation}
where $x=(x_{1},x_{2},\cdot\cdot\cdot,x_{n})$ and $\nabla=(\frac{\partial}{\partial x_{1}},\frac{\partial}{\partial x_{2}},\cdot\cdot\cdot,\frac{\partial}{\partial x_{n}})$.

Notice that
\begin{equation}\label{1-4}
\frac{\partial\Phi}{\partial t}+\frac{\partial\Phi}{\partial x_{1}}\cdot\frac{dx_{1}}{dt}+\frac{\partial\Phi}{\partial x_{2}}\cdot\frac{dx_{2}}{dt}+\cdot\cdot\cdot+\frac{\partial\Phi}{\partial x_{n}}\cdot\frac{dx_{n}}{dt}=0,
\end{equation}
thus \eqref{1-3} and (\ref{1-4}) give
\begin{equation}\label{1-5}
v_{n_{t}}=-\frac{1}{|\nabla\Phi|}\cdot\frac{\partial\Phi}{\partial t}.
\end{equation}
Combining (\ref{1-6}) and (\ref{1-5}),we obtain
\begin{equation*}
u\frac{\partial\Phi}{\partial t}=\nabla u\cdot\nabla\Phi-\chi u\nabla v\cdot\nabla\Phi\quad\text{on}\quad\partial\Omega_{t}.
\end{equation*}
At last we obtain the conditions of the free boundary $\Gamma_{t}$
\begin{equation*}
\begin{split}
-\nabla u\cdot\frac{\nabla\Phi}{|\nabla\Phi|}=ku&\quad\text{on}\quad\partial\Omega_{t},
\end{split}
\end{equation*}
and
\begin{equation*}
\begin{split}
u\frac{\partial\Phi}{\partial t}=\nabla u\cdot\nabla\Phi-\chi u\nabla v\cdot\nabla\Phi&\quad\text{on}\quad\partial\Omega_{t}.
\end{split}
\end{equation*}
Therefore the full free boundary problem reads
\begin{equation*}
\begin{cases}
u_{t}=\nabla(\nabla u-\chi u\nabla v)&\text{in}\quad\Omega_{t}\times(0,T),\\
u=0&\text{in}\quad\Omega\times(0,T)\setminus\overline{\Omega_{t}}\times(0,T),\\
-\nabla u\cdot\dfrac{\nabla\Phi}{|\nabla\Phi|}=ku&\text{on}\quad\partial\Omega_{t}\times(0,T),\\
u\dfrac{\partial\Phi}{\partial t}=\nabla u\cdot\nabla\Phi-\chi u\nabla v\cdot\nabla\Phi&\text{on}\quad\partial\Omega_{t}\times(0,T),\\
u(x,0)=u_{0}(x)&\text{in}\quad\Omega_{0},\\
0=\Delta v+u-c&\text{in}\quad\Omega\times(0,T),\\
\dfrac{\partial v}{\partial \mathbf{n}}=0&\text{on}\quad\partial\Omega\times(0,T),
\end{cases}
\end{equation*}
where $\Gamma_{t}:\Phi(x,t)=0$ is the free boundary.

%For acknowledgements section, please don't number the section, please begin it with \section*{Acknowledgements}
\section*{Acknowledgments} The authors of this paper would like to thank the referee for the comments and helpful suggestions.

% You may incorporate your references as follows in your main tex file.
% Using BibTex is not recommended but can be handled.

\medskip
% The data information below will be filled by AIMS editorial staff
Received February 2018; revised March 2018.
\medskip

\end{document}